\documentclass{amsart}
\usepackage[dvips]{color}
\usepackage{graphicx}

\usepackage{latexsym}
\usepackage{amssymb}
\usepackage{amsmath}
\usepackage{amsthm}
\usepackage{amscd}
\theoremstyle{plain}
\newtheorem{thm}{Theorem}
\newtheorem{lem}{Lemma}
\newtheorem{Bob}{Definition}
\newtheorem{prop}{Proposition}

\theoremstyle{definition}
\newtheorem{rem}{Remark}

\newtheorem{ques}{Question}
\newcommand{\Leb}{\ensuremath{\lambda}}

\newcommand{\LS}{\ensuremath{\underset{n=1}{\overset{\infty}{\cap}} \, {\underset{i=n}{\overset{\infty}{\cup}}}\,}}
\newcommand{\LSk}{\ensuremath{\underset{n=1}{\overset{\infty}{\cap}} \, {\underset{k=n}{\overset{\infty}{\cup}}}\,}}
\newcommand{\LI}{\ensuremath{\underset{n=1}{\overset{\infty}{\cup}} \, {\underset{i=n}{\overset{\infty}{\cap}}}\,}}

\newtheorem{cor}{Corollary}
\begin{document}
\author[J.\ Chaika]{Jon Chaika}
\email{jonchaika@math.uchicago.edu}
\title{Skew products over rotations with exotic properties}
\maketitle In \cite{vskew1} W. Veech  showed that a
$\mathbb{Z}_2$ skew product of an irrational rotation can be minimal
and not uniquely ergodic. He showed that this is possible when the skewing function is 
the characteristic function of an interval if and only if the irrational number is not badly
approximable. Recall that $\alpha$ is badly approximable if
${\underset{n \to \infty}{\liminf} \, n d(n\alpha, \mathbb{Z})>0}$.
Additionally, in a subsequent paper Veech \cite{vskew2} showed that
skewing a badly approximable rotation by a finite group action
 over any finite number of intervals with rational endpoints still provided a uniquely ergodic transformation. In the first remark \cite[Page 241]{vskew2} he wondered whether this was true if one dropped the assumption that the endpoints were rational (which he had already shown for the case of one skewing interval).
  The goal of this paper is to answer this question by showing that a $\mathbb{Z}_2$ skew product of a badly approximable
  rotation can be minimal but not uniquely ergodic.
   This paper will also present properties of this construction and an application to $\mathbb{Z}$ skew products of rotations.

The following two theorems and their corollaries are the main results of this paper.
\begin{thm} \label{main} There exists a $\mathbb{Z}_2$ skew product of a badly approximable rotation over two intervals that is minimal and not uniquely ergodic.
\end{thm}
\begin{cor} \label{dio} There exists a minimal, non-uniquely ergodic IET $T$ and a constant $c>0$ such that $\underset{n>0}{\inf}n \, |T^nx-x|>c$ for all $x \in [0,1)$.
\end{cor}
Corollary \ref{dio} answers a question of M. Boshernitzan.
    \begin{thm}\label{Z skew 1}
    There exists a  $\mathbb{Z}$ skew product of a badly approximable rotation over two intervals that has the following properties.
 \begin{enumerate}
 \item The full orbit of almost every point has the values taken in the second coordinate  bounded from below.
 \item Lebesgue measure is preserved but not ergodic.
 \item The ergodic measures absolutely continuous with respect to Lebesgue measure are finite.
 \end{enumerate}
 \end{thm}
 \begin{cor} \label{func} There exists $f:[0,1)\to \mathbb{Z}$ with integral 0 such that $f$ is the difference of the characteristic function of two intervals
  and for almost every $x$ $$ \underset{N\to \infty}{\liminf} \underset{i=0}{\overset{N}{\sum}} f(R^ix)>-\infty \text{ while } \underset{N\to \infty}{\limsup} \underset{i=0}{\overset{N}{\sum}} f(R^ix)=+\infty.$$
 \end{cor}
 Remark \ref{ae dense} shows that this example can be modified to construct a skew product of a rotation   over 4 intervals where the orbit of Lebesgue almost every point is dense but Lebesgue measure is not ergodic.


Some remarks on Corollary 1:
Trying to resolve the existence of such a transformation was a motivation to investigate minimal but non-uniquely ergodic skew products
 over badly approximable rotations.
This corollary helps show that the equivalence of Diophantine properties that holds true for rotations
 breaks down for IETs.  In particular, when $R$ is a rotation,
 $\underset{n>0}{\inf}n \, |R^nx-x|>c$ for all $x \in [0,1)$ iff $\{x, Rx,... ,R^nx\}$ is $\frac{c'}{n}$ dense for all $n$ and $x$.
 This condition implies unique ergodicity, because two different ergodic measures must be singular as measures.
   In general we find that it is interesting to investigate what implications of properties for rotations survive to IETs.
 It is also interesting because the previous constructions of minimal but not uniquely
  ergodic IETs \cite{satskew}, \cite{knskew}, \cite{nonue} relied implicitly on good periodic approximation.
  The property that $\underset{n>0}{\inf}n \, |T^nx-x|>c$  states that there is no particularly good periodic approximation.

  The class of examples studied in \cite{vskew1} was profitably studied in the context
   of billiards in rational polygons or flows on flat surfaces in many places.
    We mention two of them. In \cite{satskew} a closely related construction was used to show that flat surfaces of genus g
     can have g ergodic measures and \cite{cheung thesis} which showed a class of flat surfaces where
      the set of non-uniquely ergodic directions has Hausdorff dimension $\frac 1 2$
      (the appendix of that paper shows that for many flat surfaces the set of non-uniquely ergodic directions
       has Hausdorff dimension 0).
       We state without proof (it is straightforward) that dynamical systems in Theorem \ref{main}
        and Corollary \ref{dio} can be shown to arise from a billiard in a rational polygon. 
       In Section \ref{shrink} we show that the skew products we construct can provide a dynamical system  with a strange shrinking target property.

   Some remarks of Theorem 2: Skew products by rotations over intervals have been considered in many papers.
    The case when the skewing function is $\chi_{[0, \frac 1 2 )}-\chi_{[ \frac 1 2 , 1)}$ has received the most attention.
     In this case the $\mathbb{Z}$ skew product of any irrational rotation is ergodic with respect to Lebesgue measure \cite{CK}.
      For a brief discussion on similar $\mathbb{Z}$ skew products over rotations with other exotic properties see the remarks at
       the end of Section \ref{b sum}.

   Corollary \ref{func} provides an example of an ergodic transformation with a reasonable
    function (the difference of two characteristic functions) such that a positive measure set
     of points have the property that the Birkhoff sums are always greater than or equal to the expected value.
      A similar result was obtained earlier in \cite{heavy R}, where the transformation was the shift on the Thue-Morse sequence
       and the function was the characteristic function of the set of words that have $1$ in the zero$^{\text{th}}$ position minus the measure of this set.





We emphasize the fact that one can think of the construction in this paper as a limit of simpler ones which have two ergodic measures
 but are not minimal.
These transformations are chosen so that the orbit of zero becomes denser, but still stays far from being uniformly distributed.

\section{Set up}\label{setup}
First some general notation. If $S\subset \mathbb{R}$ is a set and $a\in \mathbb{R}$ then
${a+S=\{x:x-a\in S\}}$.
Any expression with an $\alpha$ is interpreted mod 1. That is $m\alpha+k$ is interpreted $m\alpha+k-\lfloor m\alpha +k \rfloor$.

Fix a badly approximable $\alpha< \frac 1 3$.
The condition that $\alpha$ is badly approximable is unnecessary, but it is the most interesting case.
 The condition that $\alpha< \frac 1 3$  is for convenience and clarity and is unimportant.

 Let $R:[0,1) \to [0,1)$ by $R(x)= x+\alpha$. (Recall that this is
 interpreted mod 1.)

 Let $q_i$ be the denominator of the $i^{th}$ convergent to $\alpha$.
Let $c_i=q_{10^i}$, $b_i=q_{2 \cdot 10^i}$.

This choice of $b_i$ and $c_i$ is for explicitness and much weaker growth conditions would suffice. See Lemmas \ref{balance restored} and \ref{balance restored 2} for the condition that one wants satisfied.

Notice that $b_i$ and $c_i$ are even. This is chosen for the sake of
convenience because $d(q_{2i}\alpha, \mathbb{Z})=q_{2i}\alpha $ mod
1.
 Let $J=[0, \underset{i=1}{\overset{\infty}{\sum}} {c_i}\alpha)$.

 Let $y=\alpha+\underset{i=1}{\overset{\infty}{\sum}} b_i \alpha $. (If one wanted to consider an $\alpha>\frac 1 3$ one could let ${y=q_3\alpha+ \underset{i=1}{\overset{\infty}{\sum}} b_i \alpha}$.)

  Let $J'=y+J= [ \alpha+ \underset{i=1}{\overset{\infty}{\sum}}  b_i\alpha, \alpha+ \underset{i=1}{\overset{\infty}{\sum}} b_i\alpha+ \underset{i=1}{\overset{\infty}{\sum}}c_i\alpha)$.

 Let $T: [0,1) \ltimes \mathbb{Z}_2 \to  [0,1) \ltimes \mathbb{Z}_2 $ by $(x,i) \to (x+\alpha, i +  \chi_{J \cup J'}(x))$.

We now define the non-minimal approximates to $T$. They play the
role that periodic approximations often play. Let $J_k= [0,
\underset{i=1}{\overset{k}{\sum}} c_i\alpha)$, $y_k= \alpha+
\underset{i=1}{\overset{k}{\sum}} b_i\alpha$.

Let $T_k:  [0,1) \ltimes \mathbb{Z}_2 \to  [0,1) \ltimes \mathbb{Z}_2 $ by $(x,i) \to (x+\alpha, i +  \chi_{J_k \cup y_k+J_k}(x))$.
Let $S_k: [0,1) \ltimes \mathbb{Z}_2 \to  [0,1) \ltimes \mathbb{Z}_2 $ by $(x,i) \to (x+\alpha, i +  \chi_{J_k \cup y_{k-1}+ J_k}(x))$.
Notice that the sequence of functions $T_1,T_2,...$ converges pointwise to $T$. Likewise  the sequence of functions $S_1,S_2,...$ converges pointwise to $T$.

 The projections onto the first and second coordinate are denoted $\pi_1$ and $\pi_2$. $\pi_1$ takes values in $[0,1)$ while $\pi_2$ takes values in either $\mathbb{Z}_2$ or $\mathbb{Z}$.

 Let $\Leb$ denote Lebesgue measure of $[0,1)$, $\Leb_2$ denote Lebesgue measure of $[0,1) \times \mathbb{Z}_2$ and let $\hat{\Leb}$ denote Lebesgue measure on $[0,1) \times \mathbb{Z}$.

 \begin{Bob} If $T:X \to X$ is a dynamical system which preserves $\mu$ then $x \in X$ is called \emph{generic}
  for $\mu$ if for any continuous function $f$ we have
   ${\underset{N \to \infty}{\lim} \frac 1 N \underset{n=0}{\overset{N-1}{\sum}}f(T^nx)=\int_X f d \mu}$.
 \end{Bob}
The interested reader may find it helpful and not too time consuming to work out what happens for $T_1,S_1,T_2,S_2$.

\section{The dynamics of $T$}
This section proves Theorem 1 and provides a description of generic points for the two ergodic measures.
 Theorem 1 could be established more quickly by Lemmas \ref{balance restored}, \ref{balance restored 2} and \ref{small shift}.
\subsection{The orbit of $(0,0)$}
In this section we describe how the orbit of $(0,0)$ changes in each successive nonminimal approximation.
 This describes the orbit of $(0,0)$ under $T$. By symmetry it also describes the orbit of $(0,1)$.

 The following lemma is important and its proof is similar to many proofs in this section and paper.
  It uses the change in behavior between consecutive non-minimal approximates to describe the behavior of $T$.

\begin{lem} \label{J first} $\underset{n=0}{\overset{N}{\sum}} \chi_J(R^n (0)) \geq \underset{n=0}{\overset{N}{\sum}} \chi_{J'} (R^n(0))$ for all $N \in \mathbb{N}$.
\end{lem}
\begin{proof}
This proof follows by induction. 
Assume
$${\underset{n=0}{\overset{N}{\sum}} \chi_{J_k}(R^n(0))\geq \underset{n=0}{\overset{N}{\sum}} \chi_{y_k+J_k }(R^n(0))} \text{ for all }N>0.$$
 We will show that $$\underset{n=0}{\overset{N}{\sum}} \chi_{J_{k+1}}(R^n(0))\geq \underset{n=0}{\overset{N}{\sum}} \chi_{y_k+J_{k+1} }(R^n(0)) \text{ for all } N>0.$$

 Observe that $J_{k+1} \backslash J_k= [\underset{i=1}{\overset{k}{\sum}} c_i\alpha, \underset{i=1}{\overset{k+1}{\sum}} c_i\alpha)$ and so $$\min\{n>0: R^n(0) \in J_{k+1} \backslash J_k\}= \underset{i=1}{\overset{k+1}{\sum}} c_i\alpha$$
  Also observe that $$(y_k+J_{k+1} )\backslash (y_k +J_k)= [\underset{i=1}{\overset{k}{\sum}} b_i\alpha+\underset{i=1}{\overset{k}{\sum}}c_i\alpha, \underset{i=1}{\overset{k}{\sum}}b_i\alpha +\underset{i=1}{\overset{k+1}{\sum}} c_i\alpha).$$
  $$\min\{n>0: R^n(0) \in y_k+J_{k+1} \backslash y_k +J_k\}= \underset{i=1}{\overset{k}{\sum}} b_i+ \underset{i=1}{\overset{k+1}{\sum}} c_i> \underset{i=1}{\overset{k+1}{\sum}} c_i.$$
  Therefore, $$\min\{n>0: R^n(0)\in (y_k+J_{k+1}) \backslash (y_k +J_k)\}>\min\{n>0:R^n(0)\in J_{k+1} \backslash
  J_k\}.$$
  Because $$R^n(0) \in (y_k+J_{k+1}) \backslash (y_k +J_k) \text{ for }n>0$$ implies that  $$R^{n-\sum_{i=1}^k b_i}(0) \in J_{k+1} \backslash J_k
\text{ and } n-\underset{i=1}{\overset{k}{\sum}} b_i>0$$ we have
$$\underset{n=0}{\overset{N}{\sum}} \chi_{J_{k+1} \backslash J_k}
R^n(0)\geq \underset{n=0}{\overset{N}{\sum}} \chi_{(y_{k+1}+J_{k+1})
\backslash (y_k+ J_k)} R^n(0) \text{ for all } N>0.$$ Therefore by
the inductive hypothesis $$\underset{n=0}{\overset{N}{\sum}}
\chi_{J_{k+1}}(R^n(0))\geq \underset{n=0}{\overset{N}{\sum}}
\chi_{y_k+J_{k+1} }(R^n(0)) \text{ for all } N>0.$$

  We conclude the proof by showing that if $$\underset{n=0}{\overset{N}{\sum}} \chi_{J_k}(R^n(0))\geq \underset{n=0}{\overset{N}{\sum}} \chi_{y_{k-1}+J_k }(R^n(0))$$ then $$\underset{n=0}{\overset{N}{\sum}}\chi_{J_{k}}(R^n(0))\geq \underset{n=0}{\overset{N}{\sum}} \chi_{y_k+J_{k} }(R^n(0)).$$
   To see this we examine where $\chi_{y_{k}+J_k }$ and $\chi_{y_{k-1}+J_k }$ differ,
  $$[y_{k-1},y_k)  \text{ and }[\underset{i=1}{\overset{k}{\sum}} c_i\alpha+\underset{i=1}{\overset{k-1}{\sum}}{b_i}\alpha,\underset{i=1}{\overset{k}{\sum}} c_i\alpha+ \underset{i=1}{\overset{k}{\sum}} b_i \alpha)= \underset{i=1}{\overset{k}{\sum}}c_i\alpha+[y_{k-1},y_k).$$
   As with the case before,
\begin{multline}
\lefteqn{\min\{n>0:R^n(0) \in [\underset{i=1}{\overset{k}{\sum}} c_i\alpha+\underset{i=1}{\overset{k-1}{\sum}}b_i\alpha,\underset{i=1}{\overset{k}{\sum}} c_i\alpha+ \underset{i=1}{\overset{k}{\sum}} b_i\alpha)\}}\\>\min\{n>0:R^n(0) \in [y_{k-1},y_k)\}.
\end{multline} The remainder of the proof of the lemma follows as above.
\end{proof}

\begin{lem} \label{balance restored} $\{n: T_k^n(0,0)\neq S_{k+1}^n(0,0)\}$ has density less than or equal to $( \underset{i=1}{\overset{k}{\sum}} b_i)c_{k+1}\alpha$.
\end{lem}
\begin{proof}Consider the set where the skewing functions for $T_k$ and $S_k$ differ: $[\underset{i=1}{\overset{k}{\sum}}c_i\alpha, \underset{i=1}{\overset{k+1}{\sum}} c_i\alpha)$ and  $[y_k+\underset{i=1}{\overset{k}{\sum}} c_i\alpha,y_k+ \underset{i=1}{\overset{k+1}{\sum}} c_i\alpha)$.
Notice that $R^n(x) \in  [\underset{i=1}{\overset{k}{\sum}} c_i\alpha, \underset{i=1}{\overset{k+1}{\sum}} c_i\alpha)$ iff  $R^{n+\sum_{i=1}^{k} b_i}(x) \in [ \underset{i=1}{\overset{k}{\sum}} c_i\alpha, \underset{i=1}{\overset{k+1}{\sum}} c_i\alpha)$.
Moreover from the proof of Lemma \ref{J first} it follows that when $x=0$ $$\min\{n >0: R^n(x) \in [\underset{i=1}{\overset{k}{\sum}} c_i\alpha, \underset{i=1}{\overset{k+1}{\sum}} c_i\alpha)\}< \min \{n> 0: R^n(x) \in y_k + [\underset{i=1}{\overset{k}{\sum}}c_i, \underset{i=1}{\overset{k+1}{\sum}}c_i).$$
Therefore any change between $$\underset{n=0}{\overset{N}{\sum}} \chi_{J_k}R^n(0)- \chi_{y_k+J_k}(R^n(0)) \text{ and } \underset{n=0}{\overset{N}{\sum}} \chi_{J_{k+1}}R^n(0)- \chi_{y_k+J_{k+1}}R^n(0)$$ is corrected after $\underset{i=1}{\overset{k}{\sum}} b_i$ steps.
The lemma follows because ${\{n: R^n(0)\in[\underset{i=1}{\overset{k}{\sum}} c_i\alpha, \underset{i=1}{\overset{k+1}{\sum}} c_i\alpha) \}}$ has density $c_{k+1}\alpha$.
\end{proof}
\begin{lem}\label{balance restored 2} $\{n: S_k^n(0,0)\neq T_{k}^n(0,0)\}$ has density less than or equal to $ (\underset{i=1}{\overset{k}{\sum}} c_i)b_{k+1}\alpha$.
\end{lem}
This follows similarly to the previous lemma by comparing $J_{k}$
and $y_{k-1}+ J_{k}$ to $J_{k}$ and $y_{k}+J_{k}$. By examining  the
gap between hits to $[y_{k-1}, y_{k})$ and hits to  $[y_{k-1}
+\underset{i=1}{\overset{k}{\sum}}c_i\alpha, y_{k}+
\underset{i=1}{\overset{k}{\sum}}c_i\alpha)$ the lemma follows.

\subsection{The behavior of typical points}
In this subsection we describe how typical points behave. In particular, we show which points leave the ergodic component of $(0,0)$ at each successive nonminimal approximation.
The results can be summed up as saying that for $k$ large enough   $\lambda_2$-almost every point behaves like either $(0,0)$ or $(0,1)$ under $S_k$ and $T_k$.

\begin{Bob} Let $U_0^{(k)}=$ $$\{(x,t): \underset {N \to \infty} {\lim}\frac 1 N
\underset{i=0}{\overset{N-1}{\sum}}
 f(T_k^i(x,t))=
\underset {N \to \infty} {\lim}\frac 1 N
\underset{i=0}{\overset{N-1}{\sum}}
f(T_k^i(0,0)) \text{ for all continuous } f\}$$
 $U_1^{(k)}=$ $$\{(x,t): \underset {N \to \infty} {\lim}\frac 1 N
\underset{i=0}{\overset{N-1}{\sum}}
 f(T_k^i(x,t))=
\underset {N \to \infty} {\lim}\frac 1 N
\underset{i=0}{\overset{N-1}{\sum}}
f(T_k^i(0,1)) \text{ for all continuous } f\}$$
 $V_0^{(k)}=$ $$\{(x,t): \underset {N \to \infty} {\lim}\frac 1 N
\underset{i=0}{\overset{N-1}{\sum}}
 f(S_k^i(x,t))=
\underset {N \to \infty} {\lim}\frac 1 N
\underset{i=0}{\overset{N-1}{\sum}}
f(S_k^i(0,0)) \text{ for all continuous } f\}$$
 $V_1^{(k)}=$ $$\{(x,t): \underset {N \to \infty} {\lim}\frac 1 N
\underset{i=0}{\overset{N-1}{\sum}}
 f(S_k^i(x,t))=
\underset {N \to \infty} {\lim}\frac 1 N
\underset{i=0}{\overset{N-1}{\sum}}
f(S_k^i(0,1)) \text{ for all continuous } f\}$$
\end{Bob}
\begin{lem} If $x \in \underset{l=1}{\overset{\sum_{i=1}^k c_i}{\cup}}R^l([y_{k-1},y_k))$ then $(x,i) \in V_j^{(k)}$ iff $(x,i)\in U_{1-j}^{(k)}$.
\end{lem}
This lemma describes the set of points that switch ergodic components between two successive nonminimal components.

Denote this set $B_k$.
\begin{proof} First we show that if $(x,i) \in V_j^{(k)}$ and $x \notin B_k$ then $(x,i) \in U_j^{(k)}$. If ${(x,i)\in V_0^{(k)}}$ and $x \notin B_k$ then $(x,i)$ is dense in the same intervals as $(0,0)$ under $S_k$.  When one examines $T_k$ one needs to look at what happens on $[y_{k-1},y_k) \cup \underset{i=1}{\overset{k}{\sum}} c_i+ [y_{k-1},y_k)$, the set where the skewing functions for $S_k$ and $T_k$ differ. The assumption that $x \notin B_k$ implies that if $n>0$ and $R^n(x) \in \underset{i=1}{\overset{k}{\sum}}c_i +[y_{k-1},y_k)$ then $n-\sum_{i=1}^k c_i>0$ and $R^{n-\sum_{i=1}^k c_i} \in [y_{k-1},y_k)$. This implies that $\pi_2(T_k^n(x,i))=\pi_2(S_k^n(x,i))$ iff $R^nx \notin B_k$. This is exactly what happens for $(0,0)$ and therefore $(x,i)\in U_0^{(k)}$.

If $(x,i)\in V_0^{(k)}$ and $x \in B_k$ then $(x,i)$ is dense in the same intervals as $(0,0)$ under $S_k$. However, $R^n(x)$ hits $\underset{i=1}{\overset{k}{\sum}}c_i+[y_{k-1},y_k)$ before it hits $[y_{k-1},y_k)$. 
 From this it follows that  $\pi_2(T_k^n(x,i))=\pi_2(S_k^n(x,i))$ iff $R^nx \in B_k$. However by the preceding paragraph $\pi_2(T_k^n(0,0))=\pi_2(S_k^n(0,0))$ iff $R^nx \notin B_k$. Therefore the copy of the line segment which contains $(0,0)$ in a given interval under $T_k$ is opposite the copy of the circle that contains $(x,i)$  under $T_k$.
\end{proof}
\begin{cor}\label{move1} $\pi_1(U_0^{(k)}\backslash V_0^{(k)})=\underset{l=1}{\overset{\sum_{j=1}^k b_j}{\cup}}R^l([\underset{j=1}{\overset{k}{\sum}} c_j,\underset{j=1}{\overset{k+1}{\sum}} c_j))$.
\end{cor}

By a similar argument we obtain
\begin{lem} \label{move2} $\pi_1(V_0^{(k+1)}\backslash U_0^{(k)})= \underset{l=1}{\overset{\sum_{j=1}^k b_j}{\cup}}R^l([\underset{i=1}{\overset{k}{\sum}}c_i,\underset{i=1}{\overset{k+1}{\sum}}c_i)$.
\end{lem}
Denote this set $C_k$.
\begin{cor}\label{all erg} $U_0^{(k)}\cup U_1^{(k)}=V_0^{(k)}\cup V_1^{(k)}=[0,1) \times \mathbb{Z}_2$.
\end{cor}
This follows by induction. The base case is straightforward and the inductive step follows by the previous lemmas which show that the points which leave one ergodic component enter the other.
\begin{cor}\label{few move} $\underset{k=1}{\overset{\infty}{\sum}} \Leb(B_k)+ \Leb(C_k)<1<\infty$.
\end{cor}
This follows from the proof of Lemma \ref{balance restored}. In particular

\begin{lem} \label{comparison} $T^n(0,0)=T_k^n(0,0)$ for all $0\leq n< \underset{i=1}{\overset{k}{\sum}}c_i$.
\end{lem}
\begin{proof} Consider $A=J\backslash J_k \cup J' \backslash (y_k +J)$. It is straightforward that ${\min\{n>0: R^n(0) \in A\}= \underset{i=1}{\overset{k}{\sum}}c_i}$.
\end{proof}






\begin{prop}\label{minimal} $T$ is minimal.
\end{prop}

\begin{lem} \label{small shift} If $0 < j< \underset{i=1}{\overset{k}{\sum}}c_i$ then $T^{j+1+b_{k+1}} (0,0)=T^j(0,0)+(b_{k+1}\alpha,1)$.
\end{lem}
\begin{proof} To see this notice that the statement is obvious for the first coordinate. For the second coordinate, by the composition of the intervals, for all $0<i< \underset{i=1}{\overset{k}{\sum}} c_i$ we have
$$R^i(0) \in J \text{ iff }R^{i+b_{k}}(0)\in J$$ and $$R^i(0) \in y+J \text{ iff }R^{i+b_{k}}(0) \in y+J.$$ Also notice that $R^{b_k}(0) \notin y+J$. Therefore $\pi_2(T^i(0,0))=\pi_2(T^{i+b_k}(0,0)) +1$ for all $1\leq i \leq \underset{i=1}{\overset{k}{\sum}}c_i.$
\end{proof}
\begin{cor} \label{dense orbit} Given $\epsilon>0$ there exists $N_{\epsilon}\in \mathbb{N}$ such that for any $r \in \mathbb{Z}$ the set $\{T^{r}(0,0),T^{r+1}(0,0),...,T^{r+N_{\epsilon}}(0,0)\}$ is $N_{\epsilon}$ dense.
\end{cor}
\begin{proof}[Proof of Proposition \ref{minimal}] This follows because $T$ is a piecewise isometry with finitely many discontinuities, no periodic points and the previous Corollary.
\end{proof}

Let $B'_k=\underset{i=1}{\overset{k^2c_k}{\cup}}R^{-i}([y_k-1,y_k))$
and
$C'_k=\underset{i=1}{\overset{k^2b_k}{\cup}}R^{-i}([\underset{i=1}{\overset{k}{\sum}}c_i,
\underset{i=1}{\overset{k+1}{\sum}}c_i))$.
\begin{lem} \label{gen} If $x \notin \LSk(B_k \cup C_k) \cup \LSk B'_k \cup C'_k$ then $$\{n>0: T^n_k(x,i) \neq T^n(x,i)\}$$ and $$\{n>0: S^n_k(x,i) \neq T^n(x,i)\}$$ have densities that go to 0 as $k$ goes to infinity.
\end{lem}
\begin{proof} By the assumption that  $x \notin \LSk(B_k \cup C_k)$ it follows that $(x,i)$ is eventually in $U_j^{(k)}$ and $V_j^{(k)}$ for all large enough $k$ and the same $j$. (That is it switches ergodic components only finitely many times.) By repeating the proof of Lemma 4 if $x \in U_j^{(k)} \cap V_j^{(k)}$ then $T_k^n(x,i) \neq S_k^n(x,i)$ iff $R^n(x) \in B_k$. Therefore if $x \notin B'k$ then $\underset{N>0}{\sup} \frac{|\{n<N:T^n_k(x,i) \neq S_k^n(x,i)\}|}{N}\leq \frac {1}{k^2+1}$. Likewise, if $x \notin C'k \cup C_k$ then $\underset{N>0}{\sup} \frac {|\{n<N:T^n_k(x,i) \neq S_{k+1}^n(x,i)\}|}{N}\leq \frac {1}{k^2+1}$. Therefore if $x \notin \underset{i=j}{\overset{\infty}{\cup}}(B_i\cup C_i \cup B'_i \cup C'_i)$ then we have$\underset{N>0}{\sup}\,\frac{|\{n<N:T_j^n(x,i) \neq T^n(x,i)\}|}{N}\leq \underset{i=j}{\overset{\infty}{\sum}}\frac 2{i^2+1}$ and the lemma follows.
\end{proof}
\begin{prop}\label{measure} $(0,0)$ is a generic point for one ergodic measure and $(0,1)$ is generic for the other.
\end{prop}
The proposition follows from the next lemma.
\begin{lem} If $x \notin \LSk B_k \cup C_k \cup B'_k \cup C'_k$ then $(x,0)$ and $(x,1)$ are generic for ergodic measures of $T$.
\end{lem}
\begin{proof} Let $f \in C([0,1] \times \mathbb{Z}_2)$. By  Corollary \ref{all erg} the $\underset{N \to \infty}{\lim}\frac{1}{N}\underset{n=1}{\overset{N}{\sum}}f(T_k^nx)=a_k$ and $\underset{N \to \infty}{\lim} \frac 1 N \underset{n=1}{\overset{N}{\sum}}f(S_kx)=a'_k$ exist. By Lemma \ref{gen} we have that the sets  $\{n:T_k^n(x,i) \neq S_k^n(x,i)\}$ and $\{n: S_{k+1}^n(x,i) \neq T_k^n(x,i)\}$ have densities that go to 0 as $k$ goes to infinity. By this fact and the fact that continuous functions on compact sets are uniformly continuous we have $\underset{k \to \infty}{\lim}a_k=a_{\infty}=\underset{k \to \infty}{\lim}a'_k$. By Lemma \ref{gen} it follows that $\underset{N \to \infty}{\lim}\frac{1}{N}\underset{n=1}{\overset{N}{\sum}}f(T^nx)= \underset{k \to \infty}{\lim} \underset{N \to \infty}{\lim}\frac{1}{N}\underset{n=1}{\overset{N}{\sum}}f(T_k^nx)=a_{\infty}$
\end{proof}
\begin{prop} $T$ has exactly two ergodic measures.
\end{prop}
\begin{proof}
All ergodic probability measures are mutually singular. Any ergodic measure must project to the first coordinate as Lebesgue measure. The two previously mentioned ergodic measures have the property that they are in opposite copies of $[0,1)$. Since the projection of each one to the first coordinate covers $[0,1)$ up to a set of Lebesgue measure zero together they cover $[0,1) \times \mathbb{Z}_2$ up to a set of Lebesgue measure zero.
\end{proof}

\subsection{Shrinking targets}\label{shrink}
This subsection shows that the previously constructed examples have an exotic shrinking target property. In particular, if $\alpha$ is badly approximable then for any $\{a_i\}_{i=1}^{\infty}$ that are decreasing with divergent sum we have $\LS B(T^ix,a_i)$ and $\LS T^{-i}B(x,a_i)$ have positive measure for every $x$. However, $\LS B(T^ix,\frac 1 i)$ has measure 1 (half the measure of the space $[0,1) \times \mathbb{Z}_2$). The result in this section is a straightforward application of known results and the previous subsections.

For ease of notation results in this section are stated in terms of $(x,0)$ and $(y,0)$. By symmetry they hold for $(x,1)$ and $(y,1)$.
\begin{prop}\label{shrink good} If $\{a_i\}_{i=1}^{\infty}$ is a decreasing sequence of positive real numbers with divergent sum then $\Leb_2 (\LS T^{-i} B((y,0),a_i))\geq 1$ for any $y \in [0,1)$.
\end{prop}
The proof requires several lemmas. The next lemma appears in \cite{shrink iet}. Its proof is included for completeness.
\begin{lem} \label{separated} Let $\epsilon>0,e>0$ and $n,t \in \mathbb{N}$. If $\{z_1,...,z_n\}\subset \mathbb{R}$ are $\frac e n$ separated and $S\subset \mathbb{R}$ is a set of measure $\epsilon$ that is the union of $t$ intervals  then  the inequality $$\Leb \left(\underset{i=1}{\overset{n}{\cup}}B(z_i,\delta) \backslash S \right) > (n-2t-\frac{n \epsilon}{e} )\delta $$ holds for any $\delta< \frac e {2n}$.
\end{lem}
\begin{proof} At most $\frac{\epsilon}{e}+2t$ of the points can lie within a $\frac e {2n}$ neighborhood of $S$. This is because an interval of length $l$ can contain at most $\lceil \frac l e \rceil$ points that are $e$ separated. Therefore all but $\frac{\epsilon}{e}+2t$ of the points $\{z_1,z_2,...,z_n\}$ have $B(z_i, \delta) \cap S=\emptyset$ and the lemma follows.
\end{proof}
The following lemma is well known and an obvious consequence of basic results relating the continued fraction expansion to homogeneous approximation.
\begin{lem}\label{cont frac} Let $C$ be the largest term in the continued fraction expansion of $\alpha$.
 $\{x, R(x),...,R^n(x)\}$ is at least $\frac {1}{2Cn}$ separated.
\end{lem}
\begin{lem}\label{some} If $\{a_i\}_{i=1}^{\infty}$ is a decreasing sequence of positive real numbers with divergent sum then $\Leb_2 (\LS T^{-i} B((y,0),a_i))>0$.
\end{lem}

\begin{proof} It suffices to show that there exist $\epsilon, M$ such that ${\Leb_2(\underset{i=N}{\overset{\infty}{\cup}} T^{-i}B((y,0),a_i))>\epsilon}$ for any $N$ and $a_i< \frac 1 {Mi}$.
 Let $C$ be the largest term in the continued fraction expansion for $\alpha$, $M>5C$ and $\epsilon< \frac 1 {5C}$. 
 By Lemma \ref{cont frac} and the definition of we have $\{(x,0),T(x,0),...,T^n(x,0)\}$ are at least $\frac 1 {2Cn}$ separated.
  By Lemma \ref{separated} it follows that $\Leb_2(\underset{i=N_1}{\overset{N_2}{\cup}} T^{-i}B((y,0),a_i))<\epsilon$ then 
  \begin{multline}
  \lefteqn{\Leb_2(\underset{i=N_2}{\overset{MN_2}{\cup}} T^{-i}B((y,0),a_i)
  \backslash \underset{i=N_1}{\overset{N_2}{\cup}} T^{-i}B((y,0),a_i))>}\\ 
  (MN_2-2N_2-MN_22C\frac{1}{5C})a_{MN_2}>N_2a_{MN_2}.
  \end{multline}
  The final inequality is a consequence of that fact that $M>5$. The Lemma follows by observing that $\underset{r=k}{\overset{\infty}{\sum}} M^{r-1}a_{M^r}$ diverges.
\end{proof}
\begin{lem} If  $\{a_i\}_{i=1}^{\infty}$ is a decreasing sequence of positive real numbers with divergent sum then $\Leb_2 (\LS T^{-i} B((y,0),a_i))$ is $T$ invariant.
\end{lem}
This lemma is obvious.
\begin{proof}[Proof of Proposition \ref{shrink good}] This follows from the previous two lemmas and the fact that the two ergodic measures of $T$ are each carried on sets of Lebesgue measure at least 1.
\end{proof}
\begin{prop} \label{shrink bad}  For almost every $y$ we have $\Leb_2 (\LS T^{-i} B((y,0),\frac 1 i) )=1$.
\end{prop}
We prove this proposition with the aid of two Lemmas. These Lemmas describe how the orbit of many points approximate other points based on previous minimal approximations. As a consequence the $T$ orbit of typical points for one ergodic measure poorly approximate typical points for the other ergodic measure.
\begin{lem} If $N> \underset{i=1}{\overset{k}{\sum}}b_i$ and  $x\notin \underset{i=1}{\overset{N}{\cup}} R^{-i}([\underset{i=1}{\overset{k}{\sum}}c_i \alpha,\underset{i=1}{\overset{k+1}{\sum}}c_i\alpha)\cup y_k+([\underset{i=1}{\overset{k}{\sum}}c_i \alpha, \underset{i=1}{\overset{k+1}{\sum}}c_i\alpha))$ then
 $$\Leb_2( \underset{i=1}{\overset{\infty}{\cup}} B(S_{k+1}^i(x,0),\frac 1 i))
\leq (c_{k+1}\alpha+\frac 2 N )\underset{i=1}{\overset{k}{\sum}}b_i+ \Leb_2 (\underset{i=1}{\overset{\infty}{\cup}} B(T_{k}^i(x,0),\frac 1 i)).$$
\end{lem}
\begin{proof} By the assumption of the lemma, $N>\underset{i=1}{\overset{k}{\sum}}b_i$ and so by Lemma \ref{move2} $x \notin C_k$. Therefore $T_k^n(x,i) \neq S_{k+1}^n(x,i)$ only when $R^n(x)\in C_k$. 
 This is $\underset{i=1}{\overset{k}{\sum}}b_i$ intervals of size $c_{k+1}\alpha$.  By the assumption of the lemma the first $n>0$ such that $R^n(x)$ is in one of these intervals is at least $N$ (because $x \notin  \underset{i=1}{\overset{N}{\cup}} R^{-i}([\underset{i=1}{\overset{k}{\sum}}c_i \alpha,\underset{i=1}{\overset{k+1}{\sum}}c_i\alpha)$), and therefore
$$ \underset{i=1}{\overset{\infty}{\cup}} B(S_{k+1}^i(x,0),\frac 1 i) \backslash \underset{i=1}{\overset{\infty}{\cup}} B(T_{k}^i(x,0),\frac 1 i)$$ is contained in a $\frac 1 N$ neighborhood of the  $\underset{i=1}{\overset{k}{\sum}}b_i$ intervals of size $c_{k+1}\alpha$ which comprise $C_k$.
\end{proof}

\begin{lem} If $N> \underset{i=1}{\overset{k}{\sum}} c_i$ and $x\notin \underset{i=1}{\overset{N}{\cup}} R^{-i}([y_{k-1},y_k) \cup(\underset{i=1}{\overset{k}{\sum}}c_i+ [y_{k-1},y_k))$ then
 $$\Leb_2 (\underset{i=1}{\overset{\infty}{\cup}} B(T_{k}^i(x,0),\frac 1 i))
\leq (b_{k}\alpha+\frac 2 N )\underset{i=1}{\overset{k}{\sum}}c_i+ \Leb_2 (\underset{i=1}{\overset{\infty}{\cup}} B(S_{k}^i(x,0),\frac 1 i)).$$
\end{lem}
The proof is similar to the preceding proof.
\begin{proof}[Proof of Proposition \ref{shrink bad}]
Notice that $(x,i)\in \LS T^{-i}B((y,0),a_i)$ iff ${(y,0) \in \LS B(T^i (x,i), a_i)}$, and thus by Fubini's Theorem Proposition \ref{shrink good} implies that it suffices to show that $\Leb_2(\LS B(T^{-i}(x,0),a_i))=1$ for almost every $x$.
By the fact that $\Leb_2$ is the sum of 2 ergodic probability measures and the fact that the sets are $T$ invariant it suffices to show that for a positive measure set of $x$ we have $\Leb_2 (\LS B(T^i(x,0), \frac 1 i))<2$.
To see that $\Leb_2(\LS B(T^i(x,0),\frac 1 i))<2$ apply the previous lemmas with $N=q_{c_{k+1}-10k}$ and $N=q_{b_k-10k}$ respectively. It is easy to see that a positive measure set of $x$ are in these sets and $\Leb_2(\LS B(T^{i}(x,0),\frac 1 i))<2$ for all such $x$.
\end{proof}

 \section{A  $\mathbb{Z}$ skew product}\label{b sum}
 In this section we prove Theorem 2. As was the case before we use nonminimal approximates to do
 this. Notations are the same as they were in Section \ref{setup} and we introduce some new notation.

Let $\hat{T}(x,j)=(x+\alpha, j+\chi_{J}-\chi_{y+J})$.

Let $\hat{T}_k(x,j)=(x+\alpha,
j+\chi_{[\underset{i=1}{\overset{k}{\sum}}c_i,
\underset{i=1}{\overset{k+1}{\sum}}c_i)}(x)-\chi_{[y_k+\underset{i=1}{\overset{k}{\sum}}c_i,y_k+\underset{i=1}{\overset{k+1}{\sum}}c_i)}(x))$.

Let $\hat{S}_k(x,i)=(x+\alpha, i+\chi_{[y_{k-1},y_k)}(x)-\chi_{[\underset{i=1}{\overset{k}{\sum}}c_i+y_{k-1},\underset{i=1}{\overset{k}{\sum}}c_i+y_k)}(x))$.

Notice that $\hat{T}_k$ is very different from $\hat T$. However
 $$F_k(x,i)=(x+\alpha, i + \pi_2(\underset{l=1}{\overset{k}{\sum}}\hat T_k(x,0))+ \pi_2(\underset{l=1}{\overset{k}{\sum}}\hat S_k(x,0)))$$ is in some senses close to $\hat T$.

\begin{lem}\label{skew} If $x \notin B_k$ the set $\{n: \pi_2(\hat{T}_k^n(x,0)) \neq 0\}$ has density $c_{k+1}\alpha \underset{i=1}{\overset{k}{\sum}}b_i$ and $\pi_{2}(\hat{T}_k^n(x,0)) \in \{0,1\}$ for all $n$.
\end{lem}
\begin{proof}
This is very similar to the proof of Lemma \ref{balance restored}.
Consider the support of the skewing function of $\hat T_k$,
$[\underset{i=1}{\overset{k}{\sum}}c_i\alpha, \underset{i=1}{\overset{k+1}{\sum}} c_i\alpha)$ and
$[y_k+\underset{i=1}{\overset{k}{\sum}} c_i\alpha,y_k+ \underset{i=1}{\overset{k+1}{\sum}} c_i\alpha)$. $
R^n(x) \in  [\underset{i=1}{\overset{k}{\sum}} c_i\alpha, \underset{i=1}{\overset{k+1}{\sum}} c_i\alpha)$
iff  $R^{n+\sum_{i=1}^{k} b_i}(x) \in [ \underset{i=1}{\overset{k}{\sum}} c_i\alpha,
 \underset{i=1}{\overset{k+1}{\sum}} c_i\alpha)$. Because, $
 c_{k+1}>\underset{i=1}{\overset{k}{\sum}}b_i$ it follows that $\pi_2(\hat T^n_k(x,0))\leq 1$ for all $n>0$.  Moreover by our assumption that $x \notin B_k$,  $$\min\{n >0: R^n(x) \in [\underset{i=1}{\overset{k}{\sum}} c_i\alpha, \underset{i=1}{\overset{k+1}{\sum}} c_i\alpha)\}< \min \{n> 0: R^n(x) \in y_k + [\underset{i=1}{\overset{k}{\sum}}c_i, \underset{i=1}{\overset{k+1}{\sum}}c_i)\}.$$
 From this it follows that $\pi_2(\hat T_k^n(x,0))\geq 0$ for all $n>0$. We obtain the result for  $n<0$ by noticing that $\hat{T}_k^{-1}(x,i)=(x-\alpha, i-\chi_J(x-\alpha)+\chi_{y+J}(x-\alpha))
$ and the fact that
 \begin{multline} \lefteqn{\min\{n >0: R^{-n}(x) \in \alpha+[\underset{i=1}{\overset{k}{\sum}} c_i\alpha, \underset{i=1}{\overset{k+1}{\sum}} c_i\alpha)\}}\\ > \min \{n> 0: R^{-n}(x) \in \alpha+y_k + [\underset{i=1}{\overset{k}{\sum}}c_i\alpha, \underset{i=1}{\overset{k+1}{\sum}}c_i\alpha)\} \end{multline}
  iff
$$\min\{n >0: R^n(x) \in [\underset{i=1}{\overset{k}{\sum}} c_i\alpha, \underset{i=1}{\overset{k+1}{\sum}} c_i\alpha)\}< \min \{n> 0: R^n(x) \in y_k + [\underset{i=1}{\overset{k}{\sum}}c_i\alpha, \underset{i=1}{\overset{k+1}{\sum}}c_i\alpha)\}.$$

The density of the set follows because $\{n: R^n(x) \in [\underset{i=1}{\overset{k}{\sum}}c_i,\underset{i=1}{\overset{k+1}{\sum}}c_i)\}$ has density $c_{k+1}\alpha$ and each hit to this interval provides $\underset{i=1}{\overset{k}{\sum}}b_i$ consecutive $j$ when $\pi_2(\hat T^j_k(x,0))=1$.
\end{proof}
We remark that if $x \in B_k$ then $\{n: \pi_2(\hat{T}_k^n(x,0)) \neq 0\}$ has density $1-c_{k+1}\alpha \underset{i=1}{\overset{k}{\sum}}b_i$ and $\pi_{2}(\hat{T}_k^n(x,0)) \in \{-1,0\}$ for all $n$.
\begin{lem} \label{skew 2} If $x \notin C_k$ the set $\{n: \pi_2(\hat{S}_k^n(0,0)) \neq 0\}$ has density $b_{k}\alpha \underset{i=1}{\overset{k}{\sum}}c_i$ and $\pi_2(\hat T_k^n(x,0)) \in \{0,1\}$ for all $n$.
\end{lem}
\begin{cor} \label{density} If $x \notin \underset{i=1}{\overset{\infty}{\cup}}B_i \cup \underset{i=1}{\overset{\infty}{\cup}} C_i \cup \LI (B'_i \cup C'_i)$ then $\{n: \pi_2(\hat T^n(x,0))=0\}$ has positive density.
\end{cor}
This is analogous to Lemma \ref{gen}.
\begin{cor} If $x \notin B_k \cup C_k$ then $\{n:F_k(x,0) \neq F_{k-1}(x,0)\}$ has density less than or equal to $b_k\alpha\underset{i=1}{\overset{k}{\sum}}c_i+c_k \alpha\underset{i=1}{\overset{k-1}{\sum}}b_i$.
\end{cor}
\begin{prop} The full orbit of $\hat{\Leb}$ almost every point is contained in a half strip.
\end{prop}
\begin{proof}
This follows from the previous corollary because the property of
being contained in a half strip is $R$  invariant. (Note that
$(x,i)$ and $(Rx,i)$ could be contained in different half strips.)
The previous corollary identifies a set of positive measure
satisfying this property. Therefore by the ergodicity of $R$ it is
true for almost every point in $[0,1) \times \{0\}$. However if the
full orbit of $(x,i)$ is contained in a half strip then the full
orbit of $(x,j)$ is  contained in a half strip too.
\end{proof}
The following lemma is a consequence of the fact that irrational rotations are uniquely ergodic.
\begin{lem} If $\mu_1$ and $\mu_2$ are absolutely continuous with respect to $\hat{\lambda}$ then there exists $k$ such that $\mu_1(A)=\mu_2(A+(0,k))$.
\end{lem}
This lemma motivates us to understand a single ergodic measure.
\begin{lem} If $x \notin \underset{k=1}{\overset{\infty}{\cup}}(B_k \cup C_k) \cup \LI (B'_k \cup C'_k)$ then $(x,0)$ is generic for $\nu$, a finite ergodic measure absolutely continuous with respect to $\hat{\lambda}$.
\end{lem}
\begin{proof} It is easy to see that the assumption
$x \notin \underset{k=1}{\overset{\infty}{\cup}}(B_k \cup C_k)$
implies that for all $k< \infty$ we have that $(x,0)$ is generic for the same ergodic measure as $(0,0)$ under $F_k$.
 Denote these finite measures $\nu_k$ and observe that $\nu_k(A) \leq \hat{\lambda}(A)$ for any measurable set $A$.
 Notice that if $f \in L_1(\hat{\lambda})$ then for any $\epsilon>0$ there exists $l_{\epsilon}:=l$ such that $\int_{[0,1) \times \mathbb{Z} \backslash [-l,l]}|f| d \hat{\lambda}<\epsilon$.
 Analogously to Proposition \ref{measure} for any
$f \in C([0,1] \times \mathbb{Z}) \cap L_1(\hat{\lambda})$
 we have that
$$\underset{N \to \infty}{\lim} \frac 1 N \underset{n=0}{\overset{N-1}{\sum}}f(\hat{T}^n(x,0))=\underset{k \to \infty}{\lim}\underset{N \to \infty}{\lim}\frac 1 N \underset{n=0}{\overset{N-1}{\sum}} f(F_k^n(x,0))= \underset{k \to \infty}{\lim}\int_{[0,1) \times \mathbb{Z}} f d\nu_k.$$ It is easy to see that Lemmas \ref{skew} and \ref{skew 2} imply that for any $f \in L_1(\hat{\lambda})$ we have $ \underset{k \to \infty}{\lim}\int_{[0,1) \times \mathbb{Z}} f d\nu_k$ exists. Corollary \ref{density} implies that $\{n: \pi_2(\hat T^n(x,0))=0\}$ has positive density and therefore because $\nu_k(A)<\hat{\lambda}(A)$ we have that $\nu$ is a finite measure.
\end{proof}
\begin{cor} Lebesgue measure is preserved but not ergodic.
\end{cor}
%
%
\begin{cor} Let $\sigma$ be an ergodic measure of $\hat T$ that is absolutely continuous with respect to Lebesgue. There exists a Borel set $U$ such that $\sigma(U^c)=0$ and $\pi_1(U)$ is almost everywhere injective.
\end{cor}This as a consequence of the finiteness of $\sigma$.
\begin{rem}\label{ae dense} Of course if the skewing function were $-\chi_J+\chi_{J'}$ then the typical point
 would have its orbit contained in a lower half strip.
  Motivated by this and the flexibility of $J$ we can pick
   another pair of skewing intervals so that we get a $\mathbb{Z}$ skew product over four intervals that has
\begin{enumerate}
\item The  orbit of almost every point is dense in $[0,1) \times \mathbb{Z}$.
\item Lebesgue measure is preserved but not ergodic.
\item The ergodic measures absolutely continuous with respect to $\hat{\lambda}$ are finite.
\end{enumerate}
For explicitness let $U=[5\alpha ,5 \alpha + \underset{i=1}{\overset{\infty}{\sum}}q_{3 \cdot 10^k}\alpha)$
Let $z= \underset{i=1}{\overset{\infty}{\sum}}q_{6 \cdot 10^k}\alpha$ and $
G:[0,1) \times  \mathbb{Z} \to [0,1) \times  \mathbb{Z}$ by $G(x,i)= (x+\alpha, i + \chi_J(x)-\chi_{y+J}(x)-\chi_U(x)+\chi_{z+U}(x))$.

This example is similar to \cite[Example 1.7]{ALMN} where a
continuous cocycle over an odometer is shown to have the property
that the orbit of Haar almost every point is dense but the Haar
measure is not ergodic.
\end{rem}
\begin{rem} These two examples provide for instances where the ergodic measures that are
absolutely continuous with respect to Lebesgue measure are finite measures. One can have the ergodic measures
 infinite and different from Lebesgue.  Let
  $$F: [0,1) \times \mathbb{Z} \to [0,1) \times \mathbb{Z} \text{ by }F(x,i) =
  (x+\alpha, i+ \chi_J-\chi_{y+J}+2\chi_{[0,\frac 1 2 )} - 2\chi_{[\frac 1 2 , 1)}).$$
   In this case $\hat{\lambda}$
   is the sum of 2 ergodic infinite measures each of which has the orbit of almost every point dense.
\end{rem}

\section{Concluding remarks}
The previous results work for any irrational rotation. For reasons
discussed in the introduction we were motivated by the case of
$\alpha$ badly approximable, in which case the base dynamics are
linearly recurrent. Perhaps the most interesting case of the results
described above is given by $\alpha$ a quadratic irrational. In this
case the rotation dynamics have an eventually self-similar induction
procedure and in special cases (when the induction procedure is
actually self-similar) arise from a substitution dynamical system.

\begin{ques} Can we do something like this for other linear recurrent IETs?
 Is there a way to do this with nice skewing function for substitution dynamical systems?
\end{ques}


 \begin{ques} (Masur) Let $f_{z}=\chi_{[0,z)}$. Given $\alpha$ what is the Hausdorff dimension of
 the set of $z$ such that the $\mathbb{Z}_2$ skew product of rotation by $\alpha$ by $f_z$ is minimal and not uniquely ergodic.
 \end{ques}

 \begin{ques} (Hooper) Does there exist $f: [0,1) \to \mathbb{Z}$ and an IET $T$ such that,
 \begin{enumerate}
 \item $\int_0^1 f(x) dx=0$.
 \item $f$ is a finite linear combination of characteristic functions of intervals.
 \item $F_k:[0,1) \times \mathbb{Z}_k \to [0,1) \times \mathbb{Z}_k$ by $F_k(x,i)=(T(x),i+f(x) mod \, k)$ is ergodic with respect to Lebesgue measure for all $k$.
 \item $F:[0,1) \times \mathbb{Z} \to [0,1) \times \mathbb{Z}$ by $F(x,i)=(T(x), i+f(x))$ is not ergodic with respect to Lebesgue measure.
 \end{enumerate}
 \end{ques}






\section{Acknowledgments} We would like to thank M. Boshernitzan, P.Hooper and H. Masur
 for helpful conversations. We were supported by Rice University's Vigre Grant and CODY while we worked on this paper.



\begin{thebibliography}{xxx}
\bibitem{ALMN} Aaronson, J; Leman\'{c}zyk, M; Mauduit, C; Nakada, H:
 Koksma's inequality and group extensions of Kronecker transformations.
   Algorithms, fractals, and dynamics (Okayama/Kyoto, 1992),  27--50, Plenum, New York, 1995.
\bibitem{c hdim} Chaika, J: Thesis.
\bibitem{shrink iet} Chaika, J: Shrinking targets for IETs: Extending a theorem of Kurzweil. arXiv:0910.2694 
\bibitem{cheung thesis} Cheung, Yitwah Hausdorff dimension of the set of nonergodic directions. With an appendix by M. Boshernitzan.  Ann. of Math. (2)  158  (2003),  no. 2, 661--678.
\bibitem{CK} Conze, J-P; Keane, M. Ergodicit\'{e} d'un flot cylindrique.
 Sminaire de Probabilits, I (Univ. Rennes, Rennes, 1976), Exp. No. 5,  7 pp. Dpt. Math. Informat., Univ. Rennes, Rennes, 1976.
\bibitem{nonue} Keane, M.
 Non-ergodic interval exchange transformations.
   Israel J. Math.  26  (1977), no. 2, 188--196
\bibitem{knskew} Keynes, Harvey B.; Newton, Dan A ``minimal'', non-uniquely ergodic interval exchange transformation.  Math. Z.  148  (1976), no. 2, 101--105.
\bibitem{heavy R} Ralston, D. Heaviness in symbolic dynamics. arXiv:0906.4003
\bibitem{satskew} Sataev, E. A.
The number of invariant measures for flows on orientable surfaces.
 (Russian)  Izv. Akad. Nauk SSSR Ser. Mat.  39  (1975), no. 4, 860--878.
\bibitem{vskew1} Veech, William A.
Strict ergodicity in zero dimensional dynamical systems and the Kronecker-Weyl theorem ${\rm mod} 2$.
 Trans. Amer. Math. Soc.  140  1969 1--33
\bibitem{vskew2} Veech, William A.
Finite group extensions of irrational rotations.
 Conference on Ergodic Theory and Topological Dynamics (Kibbutz Lavi, 1974).  Israel J. Math.  21  (1975), no. 2-3, 240--259.
\bibitem{viet} Veech, William A. Interval exchange transformations. J. Analyse Math. 33 (1978), 222--272.
\end{thebibliography}
\end{document}